\documentclass[11pt]{article}
\usepackage[english]{babel}
\usepackage{amssymb,amsmath,amsthm}
\newtheorem{theorem}{Theorem}[section]
\newtheorem{lemma}[theorem]{Lemma}
\newtheorem{proposition}[theorem]{Proposition}

{\theoremstyle{definition}\newtheorem{definition}[theorem]{Definition}}
{\theoremstyle{definition}}

\newtheorem*{thmdg}{Theorem D}
\newtheorem*{thms}{Theorem S}

\numberwithin{equation}{section}

\def\C{{\mathbb C}}
\def\N{{\mathbb N}}
\def\Z{{\mathbb Z}}
\def\R{{\mathbb R}}
\def\K{{\mathbb K}}
\def\F{{\mathcal F}}
\def\epsilon{\varepsilon}
\def\phi{\varphi}
\def\leq{\leqslant}

\def\ker{\hbox{\tt ker}\,}
\def\spann{\hbox{\tt span}\,}
\def\det{\hbox{\tt det}\,}

\title{A short proof of existence of disjoint hypercyclic operators}

\author{Stanislav Shkarin}

\date{}

\begin{document}

\maketitle

\begin{abstract} We give a short proof of existence of disjoint
hypercyclic tuples of operators of any given length on any separable
infinite dimensional Fr\'echet space. Similar argument provides
disjoint dual hypercyclic tuples of operators of any length on any
infinite dimensional Banach space with separable dual.
\end{abstract}



\section{Introduction \label{s1}}\rm

All vector spaces in this article are over the field $\K$, being
either the field $\C$ of complex numbers or the field $\R$ of real
numbers. As usual, $\Z_+$ is the set of non-negative integers and
$\N$ is the set of positive integers. Symbol $L(X)$ stands for the
space of continuous linear operators on a topological vector space
$X$, while $X'$ is the space of continuous linear functionals on
$X$. If $X$ is a normed space, then $X'$ is assumed to carry the
standard norm topology. For each $T\in L(X)$, the dual operator
$T':X'\to X'$ is defined as usual: $(T'f)(x)=f(Tx)$ for $f\in X'$
and $x\in X$. An $\F$-space is a complete metrizable topological
vector space. A locally convex $\F$-space is called a {\it
Fr\'echet} space. Recall that $x\in X$ is called a {\it hypercyclic
vector} for $T\in L(X)$ if $\{T^nx:n\in\Z_+\}$ is dense in $X$ and
$T$ is called {\it hypercyclic} if it has a hypercyclic vector. We
refer to the book \cite{bama-book} and references therein for
further information on hypercyclic operators. Recently B\'es and
Peris \cite{dh-bp} and Bernal-Gonz\'alez \cite{dh-bernal}
independently introduced the following concept.

\begin{definition}\label{dho1} Let $X$ be a topological vector space,
$m\in\N$ and ${\bf T}=(T_1,\dots,T_m)\in L(X)^m$. The $m$-tuple $\bf
T$ is called {\it disjoint hypercyclic} if there exists $x\in X$
such that
$$
O(x,{\bf T})=\{(T_1^nx,\dots,T_m^nx):n\in\Z_+\}\ \ \text{is dense in
$X^m$.}
$$
\end{definition}

Obviously, $\bf T$ is disjoint hypercyclic if and only if the
operator $T_1\oplus{\dots}\oplus T_m$ has a hypercyclic vector of
the shape $(x,\dots,x)$ for some $x\in X$. Salas \cite{dh-salas}
introduced the following concept.

\begin{definition}\label{dho2} Let $X$ be a Banach space,
$m\in\N$ and ${\bf T}=(T_1,\dots,T_m)\in L(X)^m$. Then $\bf T$ is
called {\it disjoint dual hypercyclic} if both $\bf T$ and ${\bf
T}'=(T'_1,\dots,T'_m)\in L(X')^m$ are disjoint hypercyclic.
\end{definition}

In \cite{dh-bp} and \cite{dh-bernal} interesting examples of
disjoint hypercyclic $m$-tuples are provided. B\'es, Martin and
Peris \cite{dh-bmp} and Salas \cite{dh-salas} independently
demonstrated that for any $m\in\N$, any separable infinite
dimensional Banach space supports a disjoint hypercyclic $m$-tuple
of continuous linear operators. Moreover, the construction in
\cite{dh-bmp} works also for separable infinite dimensional
Fr\'echet spaces, while the construction in \cite{dh-salas} provides
a dual disjoint hypercyclic $m$-tuple on any infinite dimensional
Banach space with separable dual. The following theorems summarize
these results.

\begin{thmdg}Let $X$ be a separable infinite dimensional Fr\'echet
space and $m\in\N$. Then there exists a disjoint hypercyclic tuple
${\bf T}=(T_1,\dots,T_m)\in L(X)^m$.
\end{thmdg}

\begin{thms}Let $X$ be an infinite dimensional Banach space with
separable dual $X'$ and $m\in\N$. Then there exists a disjoint dual
hypercyclic tuple ${\bf T}=(T_1,\dots,T_m)\in L(X)^m$.
\end{thms}

It is worth noting that Theorem~D generalizes the result of Bonet
and Peris \cite{bonper}, who demonstrated that each separable
infinite dimensional Fr\'echet space $X$ supports a hypercyclic
operator $T$. Grivaux \cite{gri} observed that hypercyclic operators
$T$ constructed in \cite{bonper} are in fact mixing and therefore
hereditarily hypercyclic. That is, for each infinite subset $A$ of
$\Z_+$ there is $x\in X$ such that $\{T^nx:n\in A\}$ is dense in
$X$. According to B\'es and Peris \cite{bp}, the direct sums of
finitely many copies of $T$ are also hypercyclic. The following
proposition summarizes this observation.

\begin{proposition}\label{sum1} For any separable infinite dimensional
Fr\'echet space $X$, there is $T\in L(X)$ such that the direct sum
$T\oplus{\dots}\oplus T=T^{\oplus m}$ of $m$ copies of $T$ is a
hypercyclic operator on $X^m$ for each $m\in\N$.
\end{proposition}

Salas \cite{saldual} proved that each infinite dimensional Banach
space $X$ with separable dual supports a dual hypercyclic operator
$T$. It is worth noting that the operator $T$ constructed by Salas
has an extra property. Namely, both $T$ and $T'$ satisfy the
hypercyclicity criterion \cite{bp} and therefore the direct sums of
finitely many copies of $T$ as well as the direct sums of finitely
many copies of $T'$ are hypercyclic. The following proposition
summarizes this observation.

\begin{proposition}\label{sum2} For any infinite dimensional
Banach space $X$ with separable $X'$, there is $T\in L(X)$ such that
the direct sum $T^{\oplus m}$ of $m$ copies of $T$ is a hypercyclic
operator on $X^m$ and the direct sum $T^{\prime\oplus m}$ of $m$
copies of $T'$ is a hypercyclic operator on $(X')^m$ for each
$m\in\N$.
\end{proposition}

Although constructions in \cite{dh-bmp,dh-salas} are interesting in
their own right, it turns out that there is a short and simple
reduction of Theorems~D and~S to the already known
Propositions~\ref{sum1} and~\ref{sum2}. The objective of this short
note is to show this.

\section{Reduction of Theorems~D and~S to Propositions~\ref{sum1}
and~\ref{sum2}}

Let $X$ be a topological vector space. Throughout this section
$GL(X)$ stands for the set of invertible linear operators $T:X\to X$
such that both $T$ and $T^{-1}$ are continuous.

\begin{lemma}\label{mainl}Let $m\in\N$, $X$ be a topological vector space,
$T_1,\dots,T_m\in L(X)$ and $(x_1,\dots,x_m)\in X^m$ be a
hypercyclic vector for $T_1\oplus{\dots}\oplus T_m$. Assume also
that $S_1,\dots,S_m\in GL(X)$ and $x\in X$ are such that $S_jx_j=x$
for $1\leq j\leq m$. Then $(x,\dots,x)\in X^m$ is a hypercyclic
vector for $R_1\oplus{\dots}\oplus R_m$, where $R_j=S_jT_jS_j^{-1}$.
In particular, $(R_1,\dots,R_m)$ is disjoint hypercyclic.
\end{lemma}

\begin{proof} Since the orbit $O$ of $(x_1,\dots,x_m)$ with respect
to $T_1\oplus{\dots}\oplus T_m$ is dense in $X^m$ and\break
$\Lambda=S_1\oplus{\dots}\oplus S_m\in GL(X^m)$, $\Lambda(O)$ is
also dense in $X^m$. Standard similarity argument shows that
$\Lambda(O)$ is exactly the orbit of $(x,\dots,x)$ with respect to
$R_1\oplus{\dots}\oplus R_m$.
\end{proof}

The following lemma is an elementary and well-known fact. We include
the proof for the sake of completeness.

\begin{lemma}\label{lcs} Let $X$ be a locally convex topological
vector space. Then the group $GL(X)$ acts transitively on
$X\setminus\{0\}$. That is, for each non-zero $x,y\in X$, there is
$S\in GL(X)$ such that $Sx=y$.
\end{lemma}

\begin{proof} Let $x$ and $y$ be non-zero vectors in $X$. If $x$ and
$y$ are not linearly independent, there exists non-zero
$\lambda\in\K$ such that $y=\lambda x$. Clearly $S=\lambda I\in
GL(X)$ and $Sx=y$. Assume now that $x$ and $y$ are linearly
independent. Using the Hahn--Banach theorem, we can find $f,g\in X'$
such that $f(x)=g(y)=1$ and $f(y)=g(x)=0$. Now we define $S\in L(X)$
by the formula $Su=u+(g-f)(u)x+(f-g)(u)y$. It is easy to see that
$S\in GL(X)$ and $Sx=y$.
\end{proof}

\begin{proof}[Proof of Theorem~D] Let $X$ be a separable infinite
dimensional Fr\'echet space and $m\in\N$. By Proposition~\ref{sum1},
there is $T\in L(X)$ such that $T^{\oplus m}\in L(X^m)$ is
hypercyclic. Let $(x_1,\dots,x_m)$ be a hypercyclic vector for
$T^{\oplus m}$. By Lemma~\ref{lcs}, we can pick $S_1,\dots,S_m\in
GL(X)$ and $x\in X$ such that $S_jx_j=x$ for $1\leq j\leq m$. By
Lemma~\ref{mainl}, $R_1\oplus{\dots}\oplus R_m$ is disjoint
hypercyclic, where $R_j=S_jTS_j^{-1}$.
\end{proof}

In order to prove Theorem~S, we need a slightly more sophisticated
version of Lemma~\ref{lcs}. First, we need the following elementary
lemma.

\begin{lemma}\label{numbers} Let $\alpha,\beta,\gamma\in\K$ be such
that $(\alpha,\beta)\neq(0,0)$ and $(\beta,\gamma)\neq(0,0)$. Then
there exist $a,b,c\in\K$ such that $a\alpha+b\beta=-1$,
$b\alpha+c\beta=1$ and $a\beta+b\gamma\neq 0$.
\end{lemma}

\begin{proof} If $\beta\neq 0$ and $\gamma\neq 0$, we set $a=0$,
$b=-\beta^{-1}$ and $c=(\alpha+\beta)\beta^{-2}$. If $\beta\neq 0$
and $\gamma=0$, we set $a=1$, $b=-(1+\alpha)\beta^{-1}$ and
$c=(\alpha+\beta+\alpha^2)\beta^{-2}$. If $\beta=0$, then
$\alpha\neq 0$ and $\gamma\neq 0$, and we set $a=-\alpha^{-1}$,
$b=\alpha^{-1}$ and $c=0$. It remains to observe that the required
conditions are satisfied.
\end{proof}

\begin{lemma}\label{lcs1} Let $X$ be a topological
vector space, $x,y\in X$ and $f,g\in X'$ be such that $f(y)=g(x)$
and $f(x)g(y)\neq f(y)g(x)$. Then there exists $S\in GL(X)$ such
that $Sx=y$ and $S'f=g$.
\end{lemma}

\begin{proof} Let $\alpha=f(x)$, $\beta=g(x)$ and $\gamma=g(y)$.
Since $f(x)g(y)\neq f(y)g(x)=g(x)^2$, $(\alpha,\beta)\neq (0,0)$ and
$(\beta,\gamma)\neq (0,0)$. By Lemma~\ref{numbers}, we can find
$a,b,c\in\K$ such that
\begin{equation}\label{equ}
af(x)+bg(x)=-1,\quad bf(x)+cg(x)=1\ \ \text{and}\ \ ag(x)+bg(y)\neq
0.
\end{equation}
Now we consider $S\in L(X)$ defined by the formula
\begin{equation}\label{S}
Su=u+af(u)x+bf(u)y+bg(u)x+cg(u)y.
\end{equation}
We shall demonstrate that $S$ satisfies all required conditions. It
is easy to see that the dual operator $S'$ acts according to the
formula
\begin{equation}\label{Sp}
S'\phi=\phi+a\phi(x)f+b\phi(y)f+b\phi(x)g+c\phi(y)g.
\end{equation}
Substituting $u=x$ into (\ref{S}) and using the equalities in
(\ref{equ}), we see that $Sx=y$. Substituting $\phi=f$ into
(\ref{Sp}) and using the equalities in (\ref{equ}), we see that
$S'f=g$. It remains to show that $S\in GL(X)$.

Since $f(x)g(y)\neq f(y)g(x)$, we see that $f,g$ are linearly
independent, $x,y$ are linearly independent and $X=L\oplus N$, where
$L=\ker f\cap \ker g$ and $N=\spann\{x,y\}$. Moreover, from
(\ref{S}) it follows that $L$ and $N$ are both invariant for $S$ and
that $Su=u$ for each $u\in L$. Thus, in order to show that $S\in
GL(X)$, it suffices to demonstrate that $S\bigr|_N$ is an invertible
operator on the 2-dimensional space $N$. The matrix $M$ of
$S\bigr|_N$ in the basis $\{x,y\}$ has shape
$$
M=\,\text{\small
$\begin{pmatrix}af(x)+bg(x)+1&af(y)+bg(y)\\
bf(x)+cg(x)&bf(y)+cg(y)+1\end{pmatrix}$}.
$$
Using the equalities in (\ref{equ}) and $f(y)=g(x)$, we can rewrite
the last matrix:
$$
M=\,\text{\small $\begin{pmatrix}0&ag(x)+bg(y)\\
1&bg(x)+cg(y)+1\end{pmatrix}$}.
$$
Hence $\det M=-(ag(x)+bg(y))\neq 0$ according to (\ref{equ}). Thus
$S\bigr|_N$ is an invertible operator on $N$ and therefore $S\in
GL(X)$.
\end{proof}

\begin{proof}[Proof of Theorem~S] Let $X$ be an infinite dimensional
Banach space such that $X'$ is separable. By Proposition~\ref{sum2},
there exists $T\in L(X)$ such that $T^{\oplus m}\in L(X^m)$ and
$T^{\prime\oplus m}\in L((X')^m)$ are hypercyclic. Pick any $x\in X$
and $f\in X'$ such that $f(x)=1$. Since the set of hypercyclic
vectors of any hypercyclic operator is dense \cite{bama-book}, we
can find a hypercyclic vector $(f_1,\dots,f_m)$ for $T^{\prime\oplus
m}$ such that functionals $f,f_1,\dots,f_m$ are linearly
independent. Once again, using the density of the set of hypercyclic
vectors of a hypercyclic operator, we can find a hypercyclic vector
$(y_1,\dots,y_m)$ for $T^{\oplus m}$ such that $f_j(y_j)\neq 0$ and
$f(y_j)f_j(x)\neq f(x)f_j(y_j)$ for $1\leq j\leq m$. Next, let
$x_j=c_jy_j$, where $c_j=\frac{1}{f_j(y_j)}$. Since $c_j$ are
non-zero constants, $(x_1,\dots,x_n)$ is also a hypercyclic vector
for $T^{\oplus m}$. Moreover, $f(x_j)f_j(x)\neq f(x)f_j(x_j)$ and
$f(x)=f_j(x_j)=1$ for $1\leq j\leq m$.  By Lemma~\ref{lcs1}, we can
find $S_j\in GL(X)$ such that $S_jx_j=x$ and $S'_jf=f_j$ for $1\leq
j\leq m$. By Lemma~\ref{mainl}, $(R_1,\dots,R_m)\in L(X)^m$ is
disjoint hypercyclic, where $R_j=S_jT_jS_j^{-1}$. Since
$R'_j=(S'_j)^{-1}T'S'_j$, $(S'_j)^{-1}\in GL(X')$ and
$(S'_j)^{-1}f_j=f$, Lemma~\ref{mainl} implies that
$(R'_1,\dots,R'_m)\in L(X')^m$ is disjoint hypercyclic. That is,
$(R_1,\dots,R_m)$ is disjoint dual hypercyclic.
\end{proof}

\small\rm

\vskip1truecm

\scshape

\noindent Stanislav Shkarin

\noindent Queens's University Belfast

\noindent Department of Pure Mathematics

\noindent University road, Belfast, BT7 1NN, UK

\noindent E-mail address: \qquad {\tt s.shkarin@qub.ac.uk}

\end{document}